\documentclass{amsart}
\usepackage[utf8]{inputenc}
\usepackage{amsthm}
\usepackage{amsfonts}
\usepackage{amsmath}
\usepackage{mathtools}
\usepackage{hyperref}
\usepackage{amssymb}
\usepackage{url}
\usepackage{fullpage}
\usepackage{comment}
\usepackage{commath}
\usepackage{graphicx}
\usepackage{subcaption}
\usepackage{standalone}
\usepackage{tikz}
\usepackage{pgfplots}
\usepackage[parfill]{parskip}
\usepackage{longtable}
\usepackage[T1]{fontenc}
\usepackage{currvita}
\pgfplotsset{compat=1.17}
\usepackage{thm-restate}

\makeatletter
\def\subsubsection{\@startsection{subsubsection}{3}%
  \z@{.5\linespacing\@plus.7\linespacing}{.1\linespacing}%
  {\normalfont\itshape}}
\makeatother

\title{$k$-Diophantine $m$-tuples in finite fields}
\author{Trajan Hammonds, Seoyoung kim, Steven J. Miller, Arjun nigam, Kyle Onghai, dishant saikia, lalit M. sharma}
\date{\today}

\theoremstyle{plain}
\newtheorem{thm}{Theorem}[section]
\newtheorem*{thm*}{Theorem}

\newtheorem{lem}[thm]{Lemma}

\theoremstyle{definition}
\newtheorem{defin}[thm]{Definition}

\theoremstyle{remark}

\numberwithin{equation}{section}

\newcommand{\Addresses}{
    \bigskip
  
    Trajan Hammonds, \textsc{Department of Mathematics, Princeton University, USA}\\ \nopagebreak
    \textit{E-mail address}: \href{mailto:trajanh@princeton.edu}{trajanh@princeton.edu}
    
    \medskip
    
    Seoyoung Kim, \textsc{Department of Mathematics and Statistics, Queen's University, Canada}\\ \nopagebreak
    \textit{E-mail address}: \href{mailto:sk206@queensu.ca}{sk206@queensu.ca}
  
    \medskip
    
    Steven J. Miller, \textsc{Department of Mathematics and Statistics, Williams College, USA}\\ \nopagebreak
    \textit{E-mail address}: \href{mailto:sjm1@williams.edu}{sjm1@williams.edu}
    
    \medskip
  
    Arjun Nigam, \textsc{University of Arizona, USA}\\ \nopagebreak
    \textit{E-mail address}: \href{mailto:arjunnigam1611@email.arizona.edu}{arjunnigam1611@email.arizona.edu}
  
    \medskip
  
    Kyle Onghai, \textsc{University of California, Los Angeles, USA}\\ \nopagebreak
    \textit{E-mail address}: \href{mailto:onghaik@g.ucla.edu}{onghaik@g.ucla.edu}
  
    \medskip
  
    Dishant Saikia, \textsc{Tezpur University, India}\\ \nopagebreak
    \textit{E-mail address}: \href{mailto:saikiadishant@gmail.com}{saikiadishant@gmail.com}
  
    \medskip
  
    Lalit M. Sharma, \textsc{University of Delhi, India}\\ \nopagebreak
    \textit{E-mail address}: \href{mailto:sharmalalit1729@gmail.com}{sharmalalit1729@gmail.com}
}

\usepackage{indentfirst}
\begin{document}
\maketitle

\begin{abstract}
    In this paper, we define a $k$-Diophantine $m$-tuple to be a set of $m$ positive integers such that the product of any $k$ distinct positive integers is one less than a perfect square. We study these sets in finite fields $\mathbb{F}_p$ for odd prime $p$ and guarantee the existence of a  $k$-Diophantine m-tuple provided $p$ is larger than some explicit lower bound. We also give a formula for the number of 3-Diophantine triples in $\mathbb{F}_p$ as well as an asymptotic formula for the number of $k$-Diophantine $k$-tuples.
\end{abstract}

\tableofcontents

\section{Introduction}
The study of Diophantine $m$-tuples can be traced to the work of Diophantus of Alexandria, and has caught the attention of numerous leading mathematicians since then. In the 3rd century, Diophantus observed that the set of four numbers: $\left\{\frac{1}{16} , \frac{33}{16}, \frac{17}{4}, \frac{105}{16}\right\}$ satisfy an interesting property that the product of any two elements in the set is one less than a rational square. This is the first example of a rational Diophantine quadruple. In the 17th century, Fermat became interested in finding integer solutions and eventually found the Diophantine quadruple $\{1, 3, 8, 120\}$. Euler discovered that the Diophantine quadruple given by Fermat can be extended to form a rational Diophantine quintuple, namely $\left\{1, 3, 8, 120, \frac{777480}{8288641}\right\}$. These sets of numbers studied by Diophantus, Fermat and Euler are now known as \textbf{Diophantine $m$-tuples}, which we define below.

\begin{defin}
    Let $S$ be a set of $m$ positive integers $\{a_1, a_2, \ldots, a_m\}$. If $a_{i}a_{j}+1$ is a perfect square for all $i, j$ such that $1\le i<j\le m$, then $S$ is a \textbf{Diophantine $m$-tuple}.   
\end{defin}

Similarly, we define a rational Diophantine $m$-tuple as follows. If $S$ is a set of $m$ positive rationals and satisfies the same condition, it is called a \textbf{rational Diophantine $m$-tuple}. For a more in-depth overview of the history of this problem, see~\cite[p. 513-519]{dickson1920history}. \par
The first important result concerning the size of Diophantine $m$-tuples was given by Baker and Davenport in 1969 \cite{bakdav}. They showed using Baker’s theory on linear forms in logarithms of algebraic numbers that if $d$ is a positive integer such that $\{1, 3, 8, d\}$ is a Diophantine quadruple, then $d$ has to be 120, implying that $\{1,3,8,120\}$ cannot be extended to a Diophantine quintuple. In 1979, Arkin, Hoggatt and Strauss showed that any Diophantine triple can be extended to a Diophantine quadruple \cite{ahs79}. In 2004, Dujella proved that there is no Diophantine sextuple and that there are at most finitely many Diophantine quintuples~\cite{duj2004diophantine}. In 2018, He, Togbé and Ziegler showed that there does not exist a Diophantine quintuple~\cite{he2018diophantine}.  

In the case of rationals, no absolute upper bound for the size of rational Diophantine $m$-tuples is known. Euler proved that there are infinitely many rational Diophantine quintuples. In 1999, Gibbs found the first rational Diophantine sextuple $\left\{\frac{11}{192} , \frac{35}{192}, \frac{155}{27}, \frac{512}{27}, \frac{1235}{48}, \frac{180873}{16}\right\}$  \cite{gib99}. In 2017, Dujella, Kazalicki, Miki\'{c}, Szikszai proved that there are infinitely many rational Diophantine sextuples \cite{ammm17}. It is not known whether there are rational Diophantine septuples. 

There are many generalizations of Diophantine $m$-tuples. One natural generalization which has been extensively studied is if we replace the number 1 in ``$a_i a_j + 1$'' with $n$. These sets are called \emph{Diophantine $m$-tuples with the property $D(n)$}. Recently, Bliznac Trebješanin and Filipin proved that there is
no $D(4)$-quintuple \cite{blif19}. Dujella,
Filipin and Fuchs proved that there does not exist a $D(-1)$-quintuple and that there are at most finitely many $D(-1)$-quadruples,
all of them containing the element 1 \cite{df05, dff07}. Recently, Bonciocat, Cipu and Mignotte proved the nonexistence of $D(-1)$-quadruples \cite{Bonciocat2020ThereIN}. 

There is an interesting connection between Diophantine $m$-tuples and elliptic curves. If $\{a,b,c\}$ are assumed to form a Diophantine triple, then in order to extend this triple to a quadruple, the task is to find an integer $x$ such that $ax+1, bx+1$ and $cx+1$ are all squares of integers. Finding a solution $x\in \mathbb{Z}$ to the three simultaneous conditions implies that there exists $y\in \mathbb{Z}$ such that
    \begin{equation}
        y^2\ =\ (ax+1)(bx+1)(cx+1);\end{equation}
this equation describes an elliptic curve. Hence, extending a Diophantine triple to a Diophantine quadruple is equivalent to finding integer solutions of the mentioned elliptic curve. A more detailed survey on Diophantine $m$-tuples and its connections with elliptic curves can be found at~\cite{dujella} and~\cite{dujella2021number}.

While most of the work on Diophantine $m$-tuples has been done over integers and rationals, Diophantine $m$-tuples may be studied over any commutative ring with identity. Studies have been made over the
ring of integers in a quadratic field (\cite{fra08}, \cite{fra09} and \cite{frso14}) by Franuši\'{c} and Soldo. In 2013, Franuši\'{c} also studied Diophantine quadruples over a cubic field \cite{fra13}. Recently, Dujella and Kazalicki studied Diophantine $m$-tuples over finite fields $\mathbb{F}_p$ where $p$ is an odd prime in~\cite{djkaz}. They proved the existence of a Diophantine $m$-tuple in $\mathbb{F}_p$ where $p$ is a prime and $p>2^{2m-2}m^2$. Using character sums, they also derive expressions for the number of Diophantine pairs, triples, and quadruples in $\mathbb{F}_p$ for given prime $p$, and provide an asymptotic formula for the number of Diophantine $m$-tuples. In recent years, there has been a lot of activity on Diophantine $m$-tuples and its generalizations. To get an extensive list of papers on Diophantine $m$-tuples, we refer the interested reader to Dujella's list of references: \url{https://web.math.pmf.unizg.hr/~duje/ref.html}. 

We study a generalization of Diophantine $m$-tuples called $k$-Diophantine $m$-tuples.

\begin{defin}
    Let $S$ be a set of $m$ positive integers $\{a_1, a_2, \ldots, a_m\}$. If $1+\prod_{j=i_1}^{i_k} a_{j}$ is a perfect square for all $i_1, \ldots, i_k\in \{1,2,\ldots,m\}$ such that $1\le i_1 < i_2 < \cdots < i_k\le m$, then $S$ is a \textbf{$k$-Diophantine $m$-tuple}.
\end{defin}

One motivation behind studying these sets is the relationship between $k$-Diophantine $k$-tuples and a well-known, open problem in number theory known as Brocard's problem. Brocard's problem asks for all integer solutions $(n, m)$ to the equation $n!+1 = m^2$. It can be clearly observed that if the elements of a $k$-Diophantine $k$-tuple are consecutive natural numbers starting from 1, then it gives a solution for Brocard's problem. Currently, there are only three known pairs of numbers solving Brocard's problem: $(4,5), (5,11), (7,71)$. Erdős conjectured that no other solutions exist. In 1993, Overholt proved that there are only finitely many solutions to Brocard's problem provided that the $abc$ conjecture is true \cite{10.1112/blms/25.2.104}. Till now, computations for $n$ up to a magnitude of $10^{15}$ have been done but yielded no further solutions for the problem.  

Moreover, just as Brocard's problem is not a trivial exercise, the same can be said of finding $k$-Diophantine $m$-tuples. Similar to the connection between Diophantine triples, i.e., 2-Diophantine triples, and elliptic curves, a connection can also be made between 3-Diophantine triples and elliptic curves. Indeed, the problem of extending a 3-Diophantine triple $\{a, b, c\}$ to a 3-Diophantine quadruple $\{a, b, c, d\}$ is equivalent to finding integer solutions of the elliptic curve
\begin{equation}
    y^2 \ = \   (abx+1)(acx+1)(bcx+1)
.\end{equation}
Hence, for even the simpler cases of $k$ and $m$, finding $k$-Diophantine $m$-tuples is already of the same complexity and importance as finding integral solutions of an elliptic curve. As no efficient, general algorithm to find integral solutions of an elliptic curve has been found yet, there is no algorithm to find the number of ways to extend a 3-Diophantine triple to a 3-Diophantine quadruple. In fact, the same can be said about the problem of extending $k$-Diophantine $k$-tuples to $k$-Diophantine $(k+1)$-tuples.

Inspired by the work of Dujella and Kazalicki~\cite{djkaz}, we studied $k$-Diophantine $m$-tuples in finite fields $\mathbb{F}_p$ where $p$ is an odd prime. We show the existence of at least one $k$-Diophantine $m$-tuple for all primes $p$ that are sufficiently large, and give a formula for the number of 3-Diophantine triples in $\mathbb{F}_p$.

In Section~\ref{prelims}, we provide results that we need to present the proofs of our new results. 
Next, we show the following theorems. 

\begin{restatable}{thm}{existence}
\label{thm:existence}
    Let $m \ge k$ be an integer. If $p>4^{\binom{m}{k-1} +1} \left(\frac{\binom{m}{k-1}}{2} + m+1\right)^2$ is a prime, then there exists at least one $k$-Diophantine $m$-tuple in $\mathbb{F}_{p}$.
\end{restatable}

Then, we prove a theorem about the number of 3-Diophantine triples in $\mathbb{F}_p$. 

\begin{restatable}{thm}{num}
\label{thm:num-3-Dio}
    Let $N_3(p)$ be the number of 3-Diophantine triples in $\mathbb{F}_p$. If $p\equiv1 \text{ mod 3}$, let $a$ be an integer such that $a\equiv2 \text{ mod 3}$ and $p=a^2+3b^2$ for some integer $b>0$. Then,
    \begin{equation}
    N_3(p)\ = \  
        \begin{cases} 
            \frac{a+1}{3} + \binom{p-1}{3}/2,& \emph{for }p\equiv 1 \pmod 3 \\
            \binom{p-1}{3}/2,& \emph{for }p\equiv 2 \pmod 3.
        \end{cases}
     \end{equation}
\end{restatable}

To do this, we need to show the following.

\begin{restatable}{thm}{cubecount} We have
\label{thm:cubecount}
    \begin{equation}
    \#\left\{(a,b,c) \in \mathbb{F}_p^3 : abc+1 \equiv 0 \pmod p \right\} \ = \  
        \begin{cases}
            (p-2)(p-3)+4, &\emph{if } p \equiv 1 \pmod {3}\\
            (p-2)(p-3), &\emph{if } p \equiv 2 \pmod {3}.
        \end{cases}
    \end{equation}
\end{restatable}

Finally, we prove the following asymptotic formula for the number of $k$-Diophantine $k$-tuples in $\mathbb{F}_p$ holds.
\begin{restatable}{thm}{asymptotic}
\label{thm:asymptotic}
    Let $N_{k}(p)$ be the number of $k$-Diophantine $k$-tuples in $\mathbb{F}_p$. Then 
    \begin{equation}
    N_{k}(p)\ \thicksim\ \frac{p^k}{k!\cdot2}+o(p^k).
    \end{equation}
\end{restatable}

\section{Preliminaries} \label{prelims}

\subsection{Legendre Symbol and Their Sums} \label{leg-symb}


First, let us define an operation from number theory known as the Legendre symbol.


We recall that if $a, p \in \mathbb{Z}$ with $p$ prime, $\gcd(a, p) = 1$, then the \textbf{Legendre Symbol}, denoted as $\Big(\frac{a}{p}\Big)$ is
\begin{equation}
        \left(\frac{a}{p}\right)\ \coloneqq\ \begin{cases}
            0 &\mbox{if } p \mid a \\
            1 &\mbox{if $a$ is a quadratic residue modulo $p$} \\
        -1 &\mbox{if $a$ is a quadratic nonresidue modulo $p$}.
        \end{cases}
    \end{equation}
Note: In the finite field $F_p$ where $p$ is an odd prime, the Legendre symbol is equivalent to the quadratic character \cite[p.~191]{lidl1997finite}.

In determining the formula for the number of 3-Diophantine triples in $\mathbb{F}_p$, we relied on two well-known sums of Legendre symbols. Consider a given polynomial $f$ with integer coefficients. The two well-known sums are special cases of the sum 
\begin{equation}
    \sum_{x=0}^{p-1} \left( \frac{f(x)}{p} \right).
.\end{equation}

If $f$ is linear, then we have the following result.

\begin{lem} \label{lem:lin}
    For arbitrary integers $a$ and $b$, and a prime $p\nmid a$, we have
    \begin{equation}
        \sum_{x=0}^{p-1}\left(\frac{ax+b}{p}\right)\ =\ 0.
    .\end{equation}
\end{lem}

\begin{proof}
     See Lemma \ref{pf-leg-lin} of Appendix \ref{sec:applegsums}.
\end{proof}

If $f$ is quadratic, then we have this next result.

\begin{lem}
\label{lem:quad}
For arbitrary integers $a, b, c$, and a prime $p$ such that $p \nmid a$, then
\begin{equation} 
    \sum_{x=0}^{p-1} \left( \frac{ax^2+bx+c}{p} \right)\ = \  
        \begin{cases}
            (p-1)\left(\frac{a}{p}\right) &\emph{if } p\mid b^2-4ac\\
            -\left(\frac{a}{p}\right) &\emph{otherwise}.
        \end{cases}
\end{equation}
\end{lem}
\begin{proof}
    See Lemma \ref{lem:pf-leg-quad} of Appendix \ref{sec:applegsums}.
\end{proof}

\subsection{Gauss's Lemma}

\begin{thm} \label{Gauss-Lem}
\emph{(Gauss)} Let $E(\mathbb{F}_p): y^2=x^3+D$ be an elliptic curve. Then for $p\equiv1\text{ mod 3}$ 
\begin{equation}\#E(\mathbb{F}_p)\ = \   
\begin{cases}
p+1+2a & \emph{if $D$ is a sextic residue mod $p$ }\\
p+1-2a  & \emph{if $D$ is cubic but not a quadratic residue mod $p$ }\\
p+1-a\pm 3b & \emph{if $D$ is a quadratic but not a cubic residue mod $p$}\\
p+1+a\pm 3b  & \emph{if $D$ is neither quadratic nor cubic residue mod $p$}\\
\end{cases}
\end{equation} 
where $a$ is an integer such that $a\equiv2 \text{ mod 3}$ and $p=a^2+3b^2$ for some integer $b>0$. For $p\equiv2\text{ mod 3}$, 
\begin{equation}\#E(\mathbb{F}_p)\ = \   p+1.\end{equation}
\end{thm}
\begin{proof}
    See ~\cite[p. 305, Thm. 4]{iros}.
\end{proof}

\subsection{Weil's Theorem and Quadratic Character Sums}

We first state Weil's theorem for the estimation of character sums; we require this result for the proof of Lemma~\ref{lidl-5.63}.

\begin{thm}[Weil]
\label{weil-estimate}
Let $\chi$ be an $n^{th}$ order non-trivial multiplicative character in the finite field $\mathbb{F}_q$. Let $f(x)$ be a degree $d$ polynomial in $\mathbb{F}_q$ such that $f(x)\neq k{g(x)}^n$ for any polynomial $g(x)$ and constant $k$ in $\mathbb{F}_q$. Then
\begin{equation}
    \abs{\sum_{x\in\mathbb{F}_q}\chi (f(x))}\ \leq\ (d-1)\sqrt{q}
.\end{equation}
\end{thm}

\begin{proof}
    c.f.~\cite[Thm. 11.23]{iwaniec2004analytic}
\end{proof}

Lemma~\ref{lidl-5.63} is needed in the proof of Lemma~\ref{lidl-5.64}.

\begin{lem}[{\cite[Ex.~5.63]{lidl1997finite}}]
\label{lidl-5.63}
    Let $a_1, \ldots, a_k$ be distinct elements of $\mathbb{F}_q$, $q$ odd, and let $\epsilon_1, \ldots, \epsilon_k$ be $k$ given integers, each of which is 1 or -1. Let $N(\epsilon_1, \ldots, \epsilon_k)$ denote the number of $c \in \mathbb{F}_q$ with $\eta(c+a_j)=\epsilon_j$ for $1 \le j \le k$, where $\eta$ is the quadratic character of $\mathbb{F}_q$. Then
    \begin{equation}
        N(\epsilon_1, \dots, \epsilon_k) \ = \  
        \frac{1}{2^k} \sum_{c\in \mathbb{F}_q}
        \left[1+\epsilon_1 \eta(c+a_1)\right] \cdots \left[1+\epsilon_k \eta(c+a_k)\right] - A
    ,\end{equation}
    where $0 \le A \le k/2$.
\end{lem}
\begin{proof}
    See Lemma \ref{lidl-5.63-pf} of Appendix \ref{sec:applegsums}.
\end{proof}

The final result we present in this section is necessary to prove the existence of $k$-Diophantine $m$-tuples in Subsection~\ref{existence}.

\begin{lem}[{\cite[Ex.~5.64]{lidl1997finite}}]
\label{lidl-5.64}
We have
    \begin{equation}
        \left| N(\epsilon_1, \ldots, \epsilon_k) - \frac{q}{2^k}\right| \ \le\
        \left(\frac{k-2}{2}+\frac{1}{2^k}\right)q^{1/2} + \frac{k}{2}
    .\end{equation}
\end{lem}

\begin{proof}
    See Lemma \ref{lidl-5.64-pf} of Appendix \ref{sec:applegsums}.
\end{proof}

We now proceed to the main results of this paper. First, we will prove that for all sufficiently large odd primes $p$, there exists at least one $k$-Diophantine $m$-tuple in $\mathbb{F}_p$.

\section{Proofs of the main results}\label{results}

\subsection{Existence of \texorpdfstring{$k$}{k}-Diophantine \texorpdfstring{$m$}{m}-tuples} \label{existence}

Here, we prove that $k$-Diophantine $m$-tuples exist for a large enough prime. First, we establish the case when $k=3$.

\begin{thm}
    Let $m \ge 3$ be an integer. If $p>2^{m^2-m-2}{(m^2+3m+4)}^2$ is a prime, then there exists at least one 3-Diophantine $m$-tuple in $\mathbb{F}_{p}$.
\end{thm}

\begin{proof}
    We prove this theorem by induction on $m$. For $m=3$ and a prime $p$ such that 
    \begin{equation}
        p\ >\ 2^{3^2-3-2}{(3^2+3(3)+4)}^2\ = \  7744
    ,\end{equation}
    we have the 3-Diophantine triple $\{2,3,4\}$ in $\mathbb{F}_p$. Indeed, $p\ge 5$ is large enough to guarantee the existence of this 3-Diophantine triple.
    Suppose that there exists at least one $3$-Diophantine $m$-tuple in $\mathbb{F}_{p}$.
    Now, we want to prove that there exists a 3-Diophantine $(m+1)$-tuple in $\mathbb{F}_{p}$ where $p$ is a prime such that $p>2^{m^2+m-2}{(m^2+5m+8)}^2$. Let us take a prime $p$ such that 
    \begin{align*}
    p &\ > \ 2^{(m+1)^2-(m+1)-2}\{(m+1)^2+3(m+1)+4)\}^2\\ 
      &\ = \   2^{m^2+m-2}{(m^2+5m+8)}^2.    
    \end{align*}
    Clearly, $p>2^{m^2-m-2}{(m^2+3m+4)}^2$.  Thus, by the induction hypothesis, there exists a 3-Diophantine $m$-tuple $\{a_1, a_2, \ldots, a_m\}$ in $\mathbb{F}_p$. Define \begin{equation}
        g\ \coloneqq\ \#\left\{x \in \mathbb{F}_p: \left(\frac{a_{i}a_{j}x+1}{p}\right)\ = \  1 \text{ where } i,j \in \mathbb{Z}, 1\leq i<j\leq m  \right\} \ = \  
        \#\left\{x \in \mathbb{F}_p : \left(\frac{x + \overline{a_{i}a_{j}}}{p}\right)\ = \  \left(\frac{\overline{a_{i}} \overline{a_{j}}}{p}\right)\right\}
    \end{equation}
     for all $i$, $j$ such that $1\leq i<j\leq m$, where $\overline{a_i}$ denotes the multiplicative inverse of $a_i$ in $\mathbb{F}_p$. We will prove that $g-(m+1)>0$, which guarantees that there exists $x \in \mathbb{F}_p, x \not\in \{0, a_1, \ldots, a_m\}$ such that $\left(\frac{a_{i}a_{j}x+1}{p}\right)=1$ with $1\leq i<j\leq m$. By choosing pairs in $\mathbb{F}_p$ in ${m \choose 2}$ ways and using Lemma \ref{lidl-5.64},
    \begin{align*}
    \abs {g-\frac{p}{2^{{m \choose 2}}}} &\ \le\ \left\{\frac{{m \choose 2} - 2}{2}+\frac{1}{2^{m \choose 2}}\right\}\sqrt{p}+\frac{{m \choose 2}}{2} \\
    g &\ \ge\  \frac{p}{2^{{m \choose 2}}}- \left\{\frac{{m \choose 2} - 2}{2}+\frac{1}{2^{m \choose 2}}\right\}\sqrt{p}-\frac{{m \choose 2}}{2}\\
    &\ \ge\ \frac{p}{2^{\frac{m(m-1)}{2}}}- \left(\frac{m(m-1)-4}{4}+\frac{1}{2^{\frac{m(m-1)}{2}}}\right)\sqrt{p}-\frac{m(m-1)}{4}.
    \end{align*}
    Since
    \begin{align*}
        &\left(\frac{m(m-1)}{4}-1+\frac{1}{2^{\frac{m(m-1)}{2}}}\right)\sqrt{p}+\frac{m(m-1)}{4}+m+1 \\
        &< \  \left(\frac{m^2-m}{4}-1+\frac{1}{2^{\frac{m(m-1)}{2}}}+\frac{1}{2^{\frac{m(m-1)}{2}+1}}\right)\sqrt{p} \\
        &=  \ \left(\frac{m^2-m}{4}-1+\frac{3}{2^{\frac{m(m-1)}{2}+1}}\right)\sqrt{p} \\
        &< \ \frac{m(m-1)\sqrt{p}}{4} \ <\ \frac{p}{2^{\frac{m(m-1)}{2}}},
    \end{align*}
    we find, $g>m+1$. So, there exists a 3-Diophantine $(m+1)$-tuple $\{a_1,\ldots, a_m, x\}$ in $\mathbb{F}_p$.
\end{proof}

We now consider the same question for arbitrary $k$.
\existence*
\begin{proof}
    We first prove the existence of a $k$-Diophantine $k$ tuple in $\mathbb{F}_p$ by using induction on $k$. Then we proceed to prove the theorem by using induction on $m$. The base case in the induction process of $m$ is the case when $m=k$ i.e the existence of a $k$-Diophantine $k$-tuple which we would have already proved. 
    \par
    Now, we prove that there exists a $k$-Diophantine $k$-tuple for $p>4^k(3k+2)^2$.
    We prove this result by induction on $k \ge 2$.
    For $p>1024$ and $k=2$, we get the Diophantine pair $\{1,3\}$ in $\mathbb{F}_p$.\par
    Assume the statement holds for $k \ge 2$. We consider a prime $p>4^{k+1} (3k+5)^2$. Since $p>4^k(3k+2)^2$, there exists a $k$-Diophantine $k$-tuple $\{a_1,a_2, \dots, a_k\}$ in $\mathbb{F}_p$. Let 
    \begin{equation}
        g\ \coloneqq\ \#\left\{x \in \mathbb{F}_p : \left(\frac{a_{1}a_{2}\dots a_{k} x+1}{p}\right)=1\right\}
    .\end{equation}
    Let $\overline{a_i}$ denote the multiplicative inverse of $a_i$ in $\mathbb{F}_p$. 
    By Lemma \ref{lidl-5.64},
    \begin{align*}
        g &\ = \   \#\left\{x \in \mathbb{F}_p : \left(\frac{x + \overline{a_{1}a_{2}\dots a_{k}}}{p}\right)\ = \  \left(\frac{\overline{a_{1}a_{2}\dots a_{k}}}{p}\right)\right\} \\
        &\ \ge\  \frac{p}{2^{\binom{k}{k}}} - \left( \frac{\binom{k}{k}-2}{2} + \frac{1}{2^{\binom{k}{k}}} \right)\sqrt{p} - \frac{\binom{k}{k}}{2}\\
        &\ =\  \frac{p-1}{2}\\
        & \ >\  k+1 \text{ for } p>4^{k+1} (3k+5)^2. 
    \end{align*}
    So, there exists at least one $x \in \mathbb{F}_p$ such that $\left(\frac{a_{1}a_{2}\dots a_{k} x+1}{p}\right)=1$ and hence we get a $(k+1)$-Diophantine $(k+1)$-tuple $\{a_1,a_2, \dots, a_k,x\}$ in $\mathbb{F}_p$. Thus, there exists a $k$-Diophantine $k$-tuple in $\mathbb{F}_p$ where $p>4^k(3k+2)^2$ and $k \ge 2$. Therefore, the base case for the induction proof holds.

    Let us now assume there exists at least one $k$-Diophantine $m$-tuple in $\mathbb{F}_p$ for $p>4^{\binom{m}{k-1} +1} \left(\frac{\binom{m}{k-1}}{2} + m+1\right)^2$. Now, we want to prove that there exists a $k$-Diophantine $(m+1)$-tuple in $\mathbb{F}_{p}$ where $p$ is a prime such that $p>4^{\binom{m+1}{k-1} +1} \left(\frac{\binom{m+1}{k-1}}{2} + m+2\right)^2$. By the induction hypothesis, since
    $p>4^{\binom{m}{k-1} +1} \left(\frac{\binom{m}{k-1}}{2} + m+1\right)^2$, there exists a $k$-Diophantine $m$-tuple $\{a_1, a_2, \ldots, a_m\}$ in $\mathbb{F}_p$. Define 
    \begin{equation}
        g\ \coloneqq\ \#\left\{x \in \mathbb{F}_p : \left(\frac{a_{i_1}a_{i_2}\dots a_{i_{k-1}}x+1}{p}\right)=1\right\}
    \end{equation} where $a_{i_1}, a_{i_2}, \ldots ,a_{i_{k-1}} \in \{a_1, a_2, \ldots, a_m\}$.
    Let $\overline{a_i}$ denote the multiplicative inverse of $a_i$ in $\mathbb{F}_p$.\\
    By Lemma \ref{lidl-5.64},
    \begin{align*}
        g &\ = \   \#\left\{x \in \mathbb{F}_p : \left(\frac{x + \overline{a_{i_1}a_{i_2}\dots a_{i_{k-1}}}}{p}\right)=\left(\frac{\overline{a_{i_1}a_{i_2}\dots a_{i_{k-1}}}}{p}\right)\right\} \\
        &\ \ge\ \frac{p}{2^{m \choose k-1}}- \left(\frac{{m \choose k-1}-2}{2}+\frac{1}{2^{m \choose k-1}}\right)\sqrt{p}-\frac{{m \choose k-1}}{2}.
    \end{align*}
    Now, we also see that $p>4^{\binom{m}{k-1} +1} \left(\frac{\binom{m}{k-1}}{2} + m+1\right)^2$ gives
    \begin{equation}\label{eqn:3.1}
        \frac{{m \choose k-1}}{2}+m+1\ <\ \frac{\sqrt{p}}{2^{{m \choose k-1} +1}}.
    \end{equation}
    Using \eqref{eqn:3.1}, we get
    \begin{align*}
        &\left(\frac{{m \choose k-1}-2}{2}+\frac{1}{2^{m \choose k-1}}\right)\sqrt{p} + \frac{{m \choose k-1}}{2}+m+1 \\
        &< \ \left(\frac{{m \choose k-1}}{2} - 1 +\frac{1}{2^{m \choose k-1}} + \frac{1}{2^{{m \choose k-1} +1}}\right)\sqrt{p} \\
        &= \ \left(\frac{{m \choose k-1}}{2} - 1 + \frac{3}{2^{{m \choose k-1} +1}}\right)\sqrt{p}\\
        &< \ \frac{{m \choose k-1}}{2}\sqrt{p}\ <\ \frac{p}{2^{m \choose k-1}}.
    \end{align*}
    Hence, we have $g>m+1$. Thus, there exists an $x \in \mathbb{F}_p, x \not\in \{0, a_1, \ldots, a_m\}$ such that $$\left(\frac{a_{i_1}a_{i_2}\dots a_{i_{k-1}}x+1}{p}\right)\ = \  1$$ where $a_{i_1}, a_{i_2}, \ldots ,a_{i_{k-1}} \in \{a_1, a_2, \ldots, a_m\}$. So, there exists a $k$-Diophantine $(m+1)$-tuple $\{a_1,\ldots, a_m, x\}$ in $\mathbb{F}_p$.
\end{proof}

\subsection{Counting 3-Diophantine Triples}

A natural question to ask is exactly how many such $k$-Diophantine $m$-tuples exist for a given $(k,m)$. The following result gives an answer for a special case.

\num*

Indeed, when we compare this formula with the results obtained computationally, we see trends in Table~\ref{table:1} and Figure~\ref{figure:num_graph} (in Appendix \ref{sec:tablesgraphs}) that give us some initial confidence in the formula's accuracy.

However, before giving the proof, we need some other results.

\subsubsection{Counting Problems}

In this subsection, we provide some lemmas that are needed to prove Theorem \ref{thm:num-3-Dio}. 

\begin{lem}
We have
    \begin{equation} 
    \#\left\{(a,b) \in \mathbb{F}_p^2 : a\neq b, ab+1 = 0 \pmod p \right\} \ = \  
        \begin{cases}
            p-3, &\emph{if } p \equiv 1 \pmod{4}\\
            p-1, &\emph{if } p \equiv 3 \pmod{4}.
        \end{cases}
    \end{equation}
\end{lem}

\begin{proof}
    First, we solve the problem without the condition that $a\neq b$. As $\mathbb{F}_p$ is a field, for each $a \in \mathbb{F}_p \setminus{\{0\}}$ there exists a unique $a^{-1} \in \mathbb{F}_p$ such that $a a^{-1} = 1$. Hence, for each $a \in \mathbb{F}_p$, take $b = -a^{-1}$.  It follows from this definition that $ab \equiv -1 \pmod p$. Since each $b$ is unique for fixed $a$, there are exactly $p-1$ pairs $(a,b) \in \mathbb{F}_p^2$ such that $ab+1=0$.
    
    Now, with the condition $a\neq b$, notice that we need only find the odd primes $p$ for which -1 is a quadratic residue. In other words, we wish to find when $\left(\frac{-1}{p}\right)=1$ where $\left(\frac{a}{b}\right)$ is the Legendre symbol. By Euler's Criterion, this is equivalent to asking when $(-1)^{\frac{p-1}{2}}=1$. This implies that -1 is a quadratic residue modulo $p$ if and only if $\frac{p-1}{2}$ is even. Since $\frac{p-1}{2}$ is even when $p \equiv 1 \pmod 4$ and odd when $p \equiv 3 \pmod 4$, we have  \begin{equation} 
    \#\left\{(a,b) \in \mathbb{F}_p^2 : a\neq b, ab+1 = 0 \pmod p \right\} \ = \  
        \begin{cases}
            p-3, &\mbox{if } p \equiv 1 \pmod{4}\\
            p-1, &\mbox{if } p \equiv 3 \pmod{4}.
        \end{cases}
    \end{equation}
\end{proof}

\cubecount*
\begin{proof}
    Consider fixing $ab=l$. We know that there is a unique choice of $c \in \mathbb{F}_p$ such that
    $lc\equiv -1 \pmod p$, namely $c = -l^{-1}$. Thus, we begin by finding
    $\#\left\{(a,b) \in \mathbb{F}_p^2 : ab\equiv l \pmod p \right\}$ where $l$ runs through the elements of
    $\mathbb{F}_p \setminus{\{0\}}$.
    
We begin by finding $\#\left\{(a,b) \in \mathbb{F}_p^2 : ab\equiv l \pmod p \right\}$. If $l$ is a
    quadratic residue modulo $p$, then there are two pairs $(a,b)$ such that $ab=l$ and $a,b$ are not distinct, for if
    $a=b=x$ is one such pair then $a=b=-x$ is the other pair. So, if $l$ is a quadratic residue, then there are $p-3$ 
    pairs. On the other hand, if $l$ is a quadratic non-residue modulo $p$, then there are $p-1$ pairs such that
    $ab=l$ and $a,b$ are distinct. Thus, we have that
    \begin{align*}
        \#\left\{(a,b) \in \mathbb{F}_p^2 : ab\equiv l \pmod p \right\} &\ =\ (p-1)\frac{p-1}{2} + (p-3)\frac{p-1}{2}\\
        &\ = \   (p-1)(p-2).
    \end{align*}
    However, the theorem statement asks a slightly different question. Now, we must consider when $c=a$ or $c=b$.
    
    \begin{description}

    \item [Case 1. $p\equiv 2 \pmod 3$]
        We want to show that there are $\frac{p+1}{2}$ residues $l$
        for which there are $p-3$ distinct triples $(a,b,c)$ and $\frac{p-3}{2}$ residues $l$ for
        which there are $p-5$ distinct triples $(a,b,c)$. This would imply that, for $p\equiv 2 \pmod 3$,
        \begin{align*}
            \#\left\{(a,b,c) \in \mathbb{F}_p^3 : abc+1 \equiv 0 \pmod p \right\} &\ =\ (p-3)\frac{p+1}{2} + (p-5)\frac{p-3}{2} \\
            &\ = \   (p-2)(p-3).
        \end{align*}
        
        The residues for which there are $p-5$ solutions satisfy $ab\equiv a^2 \equiv l \pmod p$ but not $abc\equiv 
        a^{3}\equiv -1 \pmod p$. In this case, there are two pairs such that $a=b$ and $ab\equiv l \pmod p$, i.e., 
        $(a,a)$ and $(-a,-a)$, and two pairs such that either $b=c$ and $abc\equiv ab^2\equiv -1 \pmod p$ or $a=c$
        and $abc\equiv a^2b \equiv -1$. Hence, we have $p-1-4=p-5$ solutions. We know there are $\frac{p-3}{2}$
        such residues by Euler's Criterion.
    
        The residues, $l$, for which there are $p-3$ solutions either do not satisfy $ab\equiv
        a^2\equiv l \pmod p$ or contain a pair that forms a solution of $abc \equiv a^3
        \equiv -1 \pmod p$ when extended by $c$. First, we restrict ourselves to the quadratic
        non-residues. There exist pairs $(a,c)$ or $(b,c)$, $a\neq b$ such that either $b=c$ and $abc\equiv -1
        \pmod p$ or $a=c$ and $abc\equiv -1$. Since $l$ is a quadratic non-residue, $a\neq b$. We know
        there are $\frac{p-1}{2}$ such quadratic non-residues. Now, for the cubic residues of $-1$, there are
        $q = \gcd(3, p-1)$ solutions to the congruence $x^3 \equiv -1 \pmod p$ by Euler's criterion. For $p\equiv 2 
        \pmod{3}$, $q=1$. Thus, there are $\frac{p-1}{2}+1=\frac{p+1}{2}$ residues with $p-3$ solutions to 
        $abc+1\equiv 0 \pmod p$.\\

    \item [Case 2. $p\equiv 1 \pmod 3$]
        In this case, instead of there being exactly 1 solution to $x^3\equiv -1 \pmod p$, there are
        $q = \gcd(3, p-1) = 3$ solutions. So, there are $\frac{p-1}{2}+3=\frac{p+5}{2}$ residues for
        which there are $p-3$ triples satisfying $abc+1\equiv 0 \pmod p$, and $\frac{p-1}{2}-3=\frac{p-7}{2}$ residues for which there are $p-5$ triples satisfying
        $abc+1\equiv 0 \pmod p$. Thus,
        \begin{align*}
            \#\left\{(a,b,c) \in \mathbb{F}_p^3 : abc+1 \equiv 0 \pmod p \right\} &\ =\ (p-3)\frac{p+5}{2} + (p-5)\frac{p-7}{2} \\
            &\ = \   (p-2)(p-3)+4. \qedhere
        \end{align*}
    \end{description}
    
\end{proof}

Now we present the proof of Theorem \ref{thm:num-3-Dio}. 

\begin{proof}
    We have
    \begin{equation} 12N_3(p)\ = \  \sum\left(1+\left(\frac{abc+1}{p}\right)'\right) \end{equation}
    where the sum is evaluated over non-zero and distinct $a,b,c$,
    and we have defined $\left(\frac{a}{p}\right)'=\left(\frac{a}{p}\right)$ for $a\neq0$ and $\left(\frac{a}{p}\right)'= 1$ for $a=0$. Hence
    \begin{align*}
    12N_3(p)&\ = \  \sum_{c\neq0}\,\,\,\sum_{b\neq0,c}\,\sum_{a\neq0,b,c}1 +\sum_{c\neq0}\,\,\,\sum_{b\neq0,c}\,\sum_{a\neq0,b,c}\left(\frac{abc+1}{p}\right)+ \#\left\{(a,b,c) \in \mathbb{F}_p^3 : abc+1 \equiv 0 \pmod p \right\}. \end{align*}
    The first summand will just be $(p-1)(p-2)(p-3)$, and by Lemma \ref{thm:cubecount} we already have the solution for the final summand. We use Lemmas \ref{lem:lin} and \ref{lem:quad} to evaluate the middle sum:
    \begin{equation}
        \sum_{c\neq0} \,\,\, \sum_{b\neq0,c} \, \sum_{a\neq0,b,c} \left(\frac{abc+1}{p}\right)\ = \  -\sum_{c\neq0}\,\,\,\sum_{b\neq0,c}\left(\frac{1}{p}\right)+\left(\frac{b^2c+1}{p}\right)+\left(\frac{bc^2+1}{p}\right).\end{equation}
    This gives
    \begin{align*}
        -\sum_{c\neq0}\,\,\,\sum_{b\neq0,c}1+\left(\frac{b^2c+1}{p}\right)+\left(\frac{bc^2+1}{p}\right)
        &\ = \ -\sum_{c\neq0}(p-4)-2\left(\frac{c^3+1}{p}\right)\\
        &\ = \ -(p-1)(p-4)+2\sum_{c\neq0}\left(\frac{c^3+1}{p}\right).
    \end{align*}
    Using Theorem \ref{Gauss-Lem}, we can evaluate the sum $\sum_{c\neq0}\left(\frac{c^3+1}{p}\right)$. In particular we get
    
    \begin{align*}
        \sum_{c\neq0}\left(\frac{c^3+1}{p}\right)&\ = \   
    \begin{cases}
       2a-1 & \mbox{for }p\equiv 1 \pmod 3 \\
       -1 & \mbox{for }p\equiv 2 \pmod 3
    \end{cases}
    \\
    \implies\sum_{c\neq0} \,\,\, \sum_{b\neq0,c} \, \sum_{a\neq0,b,c} \left(\frac{abc+1}{p}\right) &\ = \  \begin{cases}
        -(p-2)(p-3)+4a & \mbox{for }p\equiv 1 \pmod 3 \\
       -(p-2)(p-3) & \mbox{for }p\equiv 2 \pmod 3.
    \end{cases}
    \end{align*}
Combining this with the result from Theorem~\ref{thm:cubecount} we get the claim.
\end{proof}

As one might imagine, counting $k$-Diophantine $k$-tuples for a general $k$ is not that simple due to the complexity of the following problem: 
\
\begin{center}
    \emph{What are the total number of $k$-tuples $\{a_1,a_2,\ldots,a_k\}$ such that $a_i$ are all distinct, non-zero and}
\end{center}
\begin{equation}\prod_{i=1}^k a_i +1\ \equiv \ 0\mbox{ mod p}?\end{equation}

\subsection{Asymptotic Formula}

While a general formula of the number of $k$-Diophantine $k$-tuples is difficult, an asymptotic formula is well within reach. We describe the formula in the following result. 

\asymptotic*

\begin{proof}
We know that 
\begin{equation}k!\cdot2\cdot N_{k}(p)\ = \  \sideset{}{'}\sum\left(1+\left(\frac{1+a_{1}a_{2}\ldots a_{k}}{p}\right)' \right)\end{equation}
where the primed sum is taken over distinct and non-zero $a_1,a_2,\ldots,a_k$ and we have defined $\left(\frac{a}{p}\right)'=\left(\frac{a}{p}\right)$ for $a\neq0$ and $\left(\frac{a}{p}\right)'= 1$ for $a=0$ as before. The main term is $\sum 1=(p-1)(p-2)\ldots (p-k)=p^k+o(p^k)$. Now
\begin{equation}\sum\left(\frac{1+a_{1}a_{2}\ldots a_{k}}{p}\right)'\ = \  \sum\left(\frac{1+a_{1}a_{2}\ldots a_{k}}{p}\right)+\#\left\{(a_1,a_2,\ldots,a_k) \in \mathbb{F}_p^k : \prod_{i=1}^{k}a_i+1 \equiv 0 \pmod p \right\}.\end{equation}
Using Weil's estimate for character sums (Theorem~ \ref{weil-estimate}), we note that
\begin{equation}\sum\left(\frac{1+a_{1}a_{2}\ldots a_{k}}{p}\right)\ \leq\ p^{k-1}\sqrt{p}.\end{equation}
We also note that
\begin{equation}\#\left\{(a_1,a_2,\ldots,a_k) \in \mathbb{F}_p^k : \prod_{i=1}^{k}a_i+1 \equiv 0 \pmod p \right\}\ \leq\ (p-1)^{k-1}.\end{equation}

Hence
\begin{equation}\sum\left(\frac{1+a_{1}a_{2}\ldots a_{k}}{p}\right)'\ = \  o(p^k).\end{equation}
The result follows.
\end{proof}

\section{Concluding remarks}

In this paper, we attempted to answer two fundamental questions about $k$-Diophantine $m$-tuples:
\begin{enumerate}
    \item Given a sufficiently large prime $p$, is there always a $k$-Diophantine $m$-tuple in $\mathbb{F}_p$?
    \item Can we count the number of $k$-Diophantine $m$-tuples in $\mathbb{F}_p$ for a given prime $p$? 
    
\end{enumerate}
We give a complete answer to $(1)$ in Theorem \ref{existence}. While we were unable to answer $(2)$ for an arbitrary pair $(k,m)$, we were able to come up with an asymptotic formula for the number of $k$-Diophantine $k$-tuples for any $k$ (Theorem~\ref{thm:asymptotic}) and an explicit formula for $(k,m)=(3,3)$ (Theorem~\ref{thm:num-3-Dio}).

Some questions asked for the usual Diophantine $m$-tuples can be asked for $k$-Diophantine $m$-tuples as well. For instance, 
\begin{enumerate}
    \item Can we find a formula that, for a given $n$, allows us to count the number of $k$-Diophantine $m$-tuples with property $D(n)$, i.e. the set where $k$-wise products of distinct elements is $n$ less than a perfect square? What about about $n$ less than the $t$-th power?
    
    \item What can be said about the existence of such tuples in other commutative rings with unity, like Gaussian integers, integers, $p$-adic integers, polynomial rings etc?
\end{enumerate}

\section*{Acknowledgements}

This project was done as a part of the Polymath Jr. Program during the summer 2021. We would like to thank Professor Andrej Dujella for graciously granting us access to his book \textit{Number Theory}~\cite{dujella2021number}, and Rowan Mckee for writing the computer program for calculating the number of 3-Diophantine triples.

\appendix
\section{Proofs from the Preliminaries}\label{sec:applegsums}

For completeness we include proofs of some standard results about sums of Legendre symbols.

\begin{lem}\label{pf-leg-lin}
    For arbitrary integers $a$ and $b$, and a prime $p\nmid a$, we have
    \begin{equation}
        \sum_{x=0}^{p-1}\left(\frac{ax+b}{p}\right) \ = \   0.
    .\end{equation}
\end{lem}

\begin{proof}
    As $p\nmid a$, $ax+b$ forms a complete set of residues modulo $p$ as $x$ runs through the integers $0$ to $p-1$. Every such set contains $\frac{p-1}{2}$ quadratic residues and $\frac{p-1}{2}$ quadratic non-residues. Hence,
    \begin{equation}
        \sum_{x=0}^{p-1}\left(\frac{ax+b}{p}\right) \ = \   \frac{p-1}{2} - \frac{p-1}{2} \ = \   0.\qedhere
    \end{equation}
\end{proof}

\begin{lem}\label{lem:pf-leg-quad}
For arbitrary integers $a, b, c$, and a prime $p$ such that $p \nmid a$, then
\begin{equation} 
    \sum_{x=0}^{p-1} \left( \frac{ax^2+bx+c}{p} \right)\ = \  
        \begin{cases}
            (p-1)\left(\frac{a}{p}\right) &\emph{if } p\mid b^2-4ac\\
            -\left(\frac{a}{p}\right) &\emph{otherwise}.
        \end{cases}
\end{equation}
\end{lem}

\begin{proof}
    Notice that this sum is equivalently written as
    \begin{equation}
        \left( \frac{4a}{p} \right) \sum_{x=0}^{p-1}\left( \frac{4a^2x^2+4abx+4ac}{p} \right)
        \ = \   \left( \frac{a}{p} \right) \sum_{x=0}^{p-1} \left( \frac{{(2ax+b)}^2-(b^2-4ac)}{p} \right) \ = \   \left( \frac{a}{p} \right) S,\end{equation}
    where $S = \sum_{x=0}^{p-1} \left( \frac{{(2ax+b)}^2-(b^2-4ac)}{p} \right)$.
    
    Since the numbers $ax+b$ form a complete set of residues modulo $p$ as $x$ varies from 0 to $p-1$, we have
    \begin{equation}
        S \ = \   \sum_{l=0}^{p-1} \left( \frac{l^2-(b^2-4ac)}{p} \right)
    .\end{equation}
    It is well known that $S \equiv -1 \pmod p$ and $|S| \le p$ \cite[Ex. 10.10]{apostol1998analnt}. From this, we obtain $S = -1, p-1$. If $S = p-1$, then $p-1$ terms in $S$ must take the value 1 and there is exactly one term, i.e., when $l=l'$, that equals 0. As this $l'$ must satisfy both $p \mid {l'}^2-(b^2-4ac)$ and $p \mid {(-l')}^2-(b^2-4ac)$, it follows that $l' = 0$ and $p \mid b^2-4ac$. Conversely, if $p \mid b^2-4ac$, then
    \begin{equation}
        S \ = \   \sum_{l=0}^{p-1} \left( \frac{l^2}{p} \right) \ = \   0 + 1\cdot p-1 \ = \   p-1.\end{equation}
    Hence, $S = -1$ if and only if $p \nmid b^2-4ac$. Thus,
    \begin{equation} 
        \sum_{x=0}^{p-1} \left( \frac{ax^2+bx+c}{p} \right)\ = \  
        \begin{cases}
            (p-1)\left(\frac{a}{p}\right) &\mbox{if } p\mid b^2-4ac\\
            -\left(\frac{a}{p}\right) &\mbox{otherwise}.
        \end{cases}
    \qedhere
    \end{equation}
\end{proof}

\begin{lem}[{\cite[Ex.~5.63]{lidl1997finite}}]
\label{lidl-5.63-pf}
    Let $a_1, \ldots, a_k$ be distinct elements of $\mathbb{F}_q$, $q$ odd, and let $\epsilon_1, \ldots, \epsilon_k$ be $k$ given integers, each of which is 1 or -1. Let $N(\epsilon_1, \ldots, \epsilon_k)$ denote the number of $c \in \mathbb{F}_q$ with $\eta(c+a_j)=\epsilon_j$ for $1 \le j \le k$, where $\eta$ is the quadratic character of $\mathbb{F}_q$. Then
    \begin{equation}
        N(\epsilon_1, \dots, \epsilon_k) \ = \  
        \frac{1}{2^k} \sum_{c\in \mathbb{F}_q}
        \left[1+\epsilon_1 \eta(c+a_1)\right] \cdots \left[1+\epsilon_k \eta(c+a_k)\right] - A
    ,\end{equation}
    where $0 \le A \le k/2$.
\end{lem}

\begin{proof}
    Notice that, for fixed $c$, if $\eta(c+a_j)=\epsilon_j$, then $\epsilon_{j}\eta(c+a_j)=1$; otherwise, $\epsilon_{j}\eta(c+a_j)=-1$ for $c+a_j\neq 0$ or $\eta(c+a_j)=0$.
    This implies that, if $\eta(c+a_j)=\epsilon_j$ for all $j\in \mathbb{N}$ such that $1\le j\le k$ and $c$ fixed, then
    \begin{equation}
        \left[1+\epsilon_1 \eta(c+a_1)\right] \cdots \left[1+\epsilon_k \eta(c+a_k)\right] \ = \   2^k
    .\end{equation}
    Otherwise, 
    \begin{equation}
        \left[1+\epsilon_1 \eta(c+a_1)\right] \cdots \left[1+\epsilon_k \eta(c+a_k)\right] \ = \   0
    \end{equation}
    if $\epsilon_i \eta(c+a_i)=-1$ for some $i\in \mathbb{N}$ such that $1\le i\le k$, or
    \begin{equation}
        \left[1+\epsilon_1 \eta(c+a_1)\right] \cdots \left[1+\epsilon_k \eta(c+a_k)\right] \ = \   2^{k-1}
    \end{equation}
    if $\epsilon_i \eta(c+a_i)=0$ for some $i\in \mathbb{N}$ and $\eta(c+a_j)=1$ for all $j\neq i$. Note that there is at most one $a_i$ with the property that $\eta(c+a_i)=0$ since $a_1, \ldots, a_k$ are distinct and $c$ is constant. Thus, we have
    \begin{equation}
        N(\epsilon_1, \dots, \epsilon_k) \ = \  
        \frac{1}{2^k} \sum_{c\in \mathbb{F}_q}
        \left[1+\epsilon_1 \eta(c+a_1)\right] \cdots \left[1+\epsilon_k \eta(c+a_k)\right] - A
    ,\end{equation}
    where $0 \le A \le k/2$.
\end{proof}

\begin{lem}[{\cite[Ex.~5.64]{lidl1997finite}}]
\label{lidl-5.64-pf}
We have
    \begin{equation}
        \left| N(\epsilon_1, \ldots, \epsilon_k) - \frac{q}{2^k}\right| \ \le\
        \left(\frac{k-2}{2}+\frac{1}{2^k}\right)q^{1/2} + \frac{k}{2}
    .\end{equation}
\end{lem}

\begin{proof}
    First, we expand the product in the expression for $N(\epsilon_1, \ldots, \epsilon_k)$ given above. 
    
    \begin{multline} \label{quad_expand}
        [1+\epsilon_1\eta(c+a_1)]\cdots [1+\epsilon_k\eta(c+a_k)] \ = \   1+\sum_{i=1}^{k} \epsilon_i \eta(c+a_i) + \sum_{j\neq i}^k \sum_{i=1}^k \epsilon_i \epsilon_j \eta(c+a_i)\eta(c+a_j) \\ + \cdots + \epsilon_1 \epsilon_2 \cdots \epsilon_k \eta(c+a_1) \eta(c+a_2) \cdots \eta(c+a_k)
    .\end{multline}
    
    By the multiplicative nature of the quadratic character, \ref{quad_expand} is equivalent to
    \begin{multline*}
         1+\sum_{i=1}^{k} \epsilon_i \eta(c+a_i) + \sum_{j\neq i}^k \sum_{i=1}^k \epsilon_i \epsilon_j \eta[(c+a_i)(c+a_j)] + \cdots + \epsilon_1 \epsilon_2 \cdots \epsilon_k \eta[(c+a_1) (c+a_2)\cdots (c+a_k)].
    \end{multline*}
    
    We see that the product expands into a sum of quadratic characters of functions in $c$ of various degrees. More specifically, we find that there are $\binom{k}{i}$ functions of degree $i$ where $1 \le i \le k$. By Theorem~\ref{weil-estimate}, we can show that
    \begin{align*}
        \left| \sum_{c\in \mathbb{F}_q} \sum_{d=1}^k \eta(f_d(x)) \right| &\ \le\ \sum_{d=1}^k \binom{k}{d}(d-1)\sqrt{q} \\
        &\ = \   \sqrt{q} \left( \sum_{d=1}^k \binom{k}{d}d - \sum_{d=1}^k \binom{k}{d} \right)\\
        &\ = \   \left(k 2^{k-1} - 2^{k-1}\right) \sqrt{q}
    ,\end{align*}
    where $f_d$ are all the functions of degree $d$ for which the quadratic character is evaluated. It follows that
    \begin{equation}
        \left| N(\epsilon_1, \ldots, \epsilon_k) - \frac{q}{2^k}\right| \le
        \left(\frac{k-2}{2}+\frac{1}{2^k}\right)q^{1/2} + \frac{k}{2}.\qedhere\end{equation}
\end{proof}

\newpage
\section{Tables and Graphs}\label{sec:tablesgraphs}

The tables and graphs shown are shown here to offer some computational verification of the formula presented in Theorem~\ref{thm:num-3-Dio} and demonstrate what the program we used is capable of.  As the formula predicts, we see in Table~\ref{table:1} and Figure~\ref{figure:num_graph} a positive quadratic trend in the number of 3-Diophantine triples against the size of $\mathbb{F}_p$.  The reason we decided not to plot values of $p$ past 200 was the computational cost of doing so and that one can see this trend for $p<300$. The program was not only able to give the number of $k$-Diophantine $m$-tuples in $\mathbb{F}_p$, but also the explicit $m$-tuples themselves, as seen in Table~\ref{table:list}.

\begin{table}[h]
\caption{Number of 3-Diophantine triples in $\mathbb{F}_p$ for various primes $p$.}
    \centering
    \begin{tabular}{ | c || c | c | c | c | }
        \hline
        $p$ & $p\equiv 1, 2 \pmod 3$ & $N_3(p)$ & $a$ (if $p\equiv 1 \pmod 3$) & Error term $(a+1)/3$\\ \hline
        5 & 2 & 2 & - & -\\ \hline
        7 & 1 & 11 & 2 & 1\\ \hline
        11 & 2 & 60 & - & -\\ \hline
        13 & 1 & 110 & -1 & 0\\ \hline
        17 & 2 & 280 & - & -\\ \hline
        19 & 1 & 407 & -4 & -1\\ \hline
        23 & 2 & 770 & - & -\\ \hline
        29 & 2 & 1638 & - & -\\ \hline
        31 & 1 & 2031 & 2 & 1\\ \hline
        37 & 1 & 3572 & 5 & 2\\ \hline
        41 & 2 & 4940 & - & -\\ \hline
        43 & 1 & 5739 & -4 & -1\\ \hline
        101 & 2 & 80850 & - & -\\ \hline
        229 & 1 & 97472 & 11 & 4\\ \hline
    \end{tabular}
\label{table:1}
\end{table}

\begin{figure}[htb]
  \centering
  \caption{Graph of number of 3-Diophantine triples in $\mathbb{F}_p$ for various $p$.}
  \includegraphics[width=\linewidth]{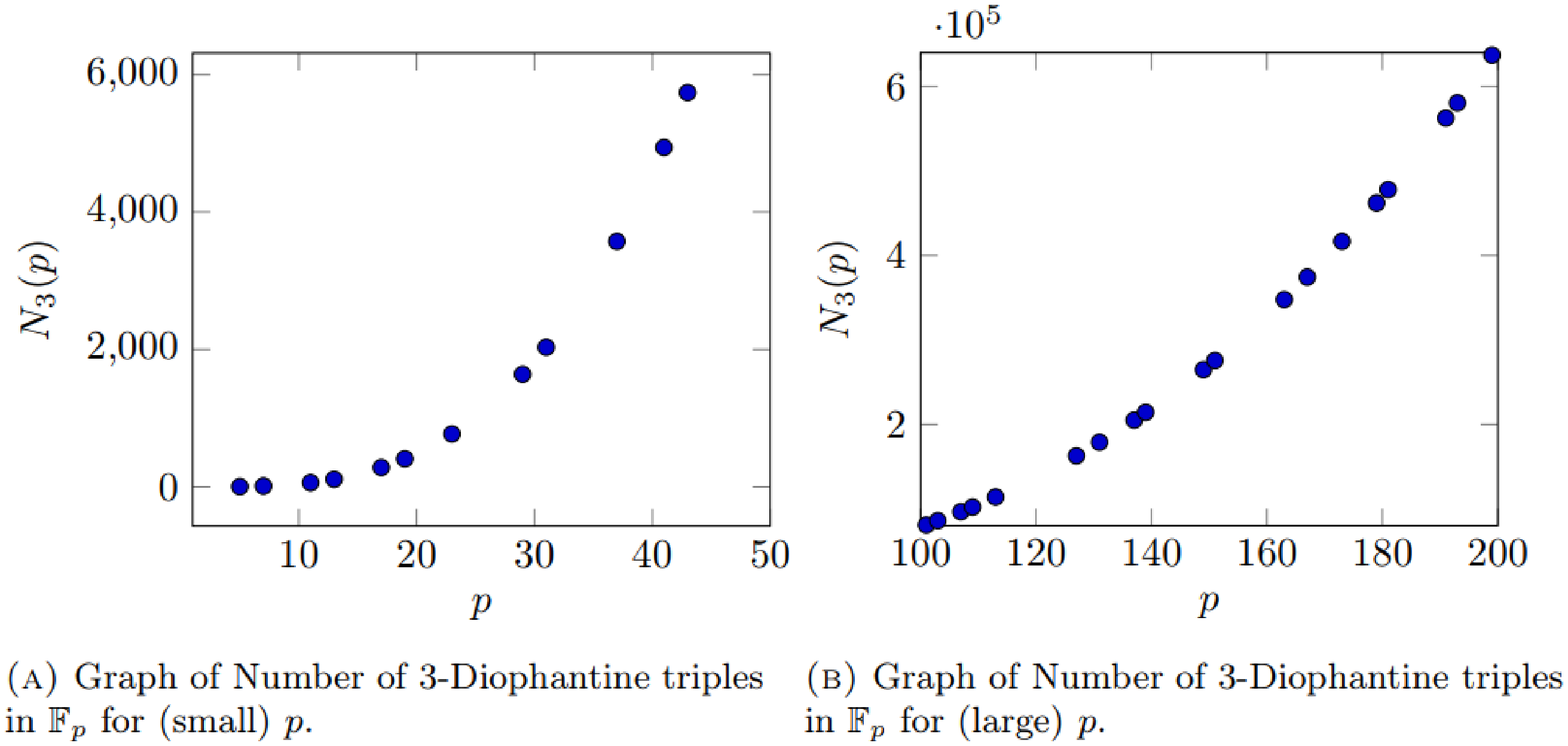}
  
  \label{figure:num_graph}
\end{figure}

\begin{figure}[htb]
  \centering
  \caption{Log-log graph of number of 3-Diophantine triples in $\mathbb{F}_p$ for various $p$.}
  \includegraphics[width=\linewidth]{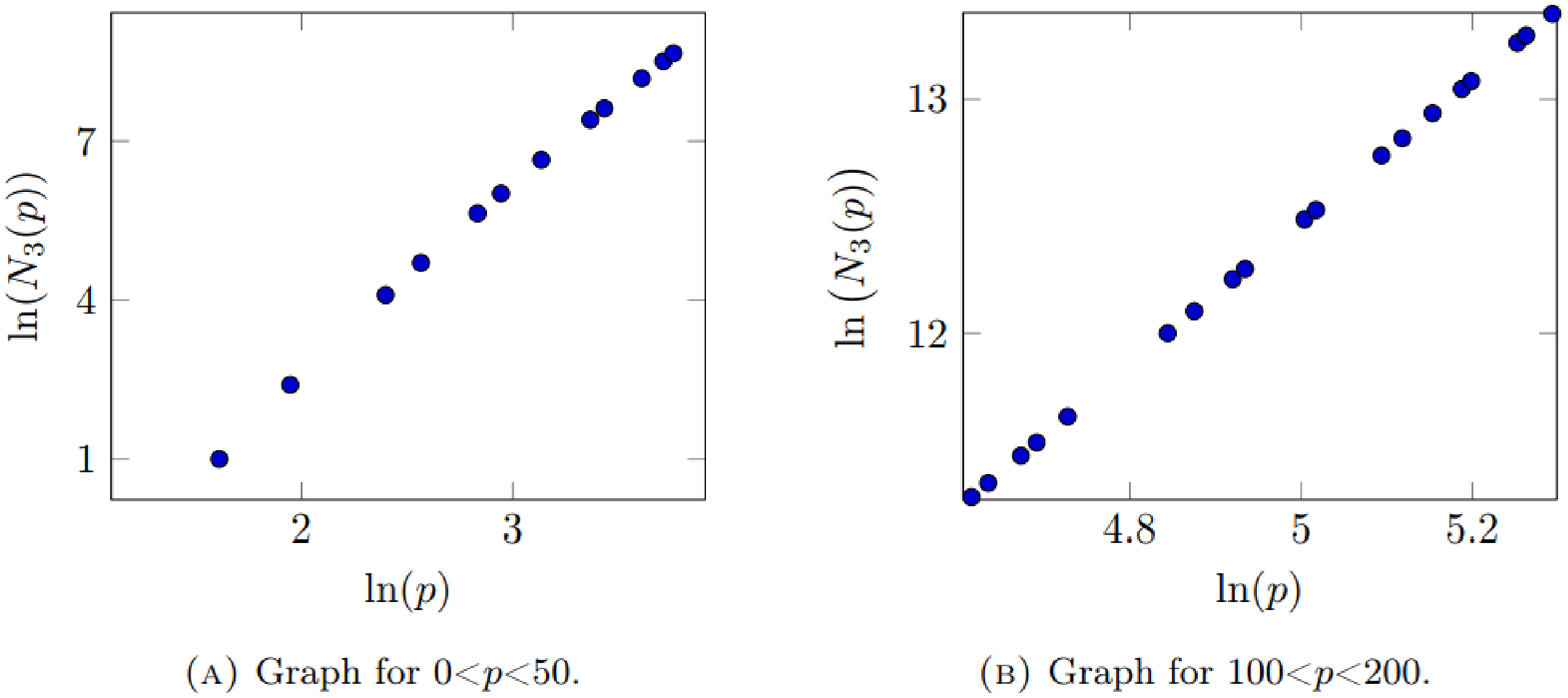}
  
  \label{figure:loglog_graph}
\end{figure}

\newpage
\begin{longtable}{ | c | c | c | c | c | }
    \caption{List of all 3-Diophantine Quadruples in $\mathbb{F}_{23}$.}
        \label{table:list}\\
        \hline
        
        (1, 2, 4, 6)& (1, 2, 4, 20)& (1, 2, 6, 14)& (1, 2, 11, 12)& (1, 2, 13, 15)\\\hline (1, 2, 13, 19)& (1, 2, 15, 20)& (1, 2, 17, 19)& (1, 2, 19, 20)& (1, 3, 4, 10)\\ \hline (1, 3, 4, 21)& (1, 3, 5, 8)& (1, 3, 8, 15)& (1, 3, 15, 18)& (1, 3, 16, 19)\\  \hline (1, 3, 17, 19)& (1, 3, 19, 21)& (1, 4, 6, 12)& (1, 4, 7, 20)& (1, 4, 10, 21)\\ \hline (1, 4, 17, 18)& (1, 4, 17, 21)& (1, 4, 19, 20)& (1, 5, 6, 9)& (1, 5, 6, 20)\\ \hline (1, 5, 8, 9)& (1, 5, 16, 19)& (1, 6, 8, 20)& (1, 6, 9, 21)& (1, 6, 12, 21)\\ \hline (1, 6, 19, 20)& (1, 6, 19, 21)& (1, 6, 21, 22)& (1, 7, 9, 11)& (1, 7, 9, 18)\\ \hline (1, 7, 10, 12)& (1, 7, 11, 13)& (1, 7, 13, 20)& (1, 8, 9, 21)& (1, 8, 10, 13)\\ \hline (1, 8, 10, 21)& (1, 8, 15, 20)& (1, 8, 20, 22)& (1, 9, 11, 13)& (1, 9, 14, 17)\\ \hline (1, 9, 17, 18)& (1, 10, 12, 14)& (1, 10, 12, 16)& (1, 10, 14, 15)& (1, 11, 12, 21)\\ \hline (1, 11, 13, 16)& (1, 11, 13, 19)& (1, 11, 16, 22)& (1, 11, 17, 19)& (1, 11, 17, 21)\\ \hline (1, 11, 21, 22)& (1, 12, 16, 22)& (1, 12, 21, 22)& (1, 14, 15, 17)& (1, 14, 15, 18)\\ \hline (1, 14, 16, 18)& (1, 15, 20, 22)& (1, 16, 18, 22)& (1, 18, 20, 22)& (2, 3, 4, 21)\\ \hline (2, 3, 4, 22)& (2, 3, 5, 8)& (2, 3, 5, 14)& (2, 3, 8, 12)& (2, 3, 8, 22)\\ \hline (2, 3, 9, 20)& (2, 3, 14, 20)& (2, 3, 19, 22)& (2, 4, 5, 13)& (2, 4, 5, 15)\\ \hline (2, 4, 6, 22)& (2, 4, 9, 15)& (2, 4, 9, 20)& (2, 4, 10, 15)& (2, 4, 13, 21)\\ \hline (2, 4, 20, 21)& (2, 5, 7, 10)& (2, 5, 7, 14)& (2, 5, 7, 16)& (2, 5, 8, 14)\\ \hline (2, 5, 12, 15)& (2, 5, 13, 16)& (2, 5, 13, 21)& (2, 5, 16, 21)& (2, 6, 7, 16)\\ \hline (2, 6, 11, 22)& (2, 6, 16, 21)& (2, 7, 9, 14)& (2, 7, 9, 15)& (2, 7, 9, 18)\\ \hline (2, 7, 10, 15)& (2, 7, 12, 17)& (2, 7, 16, 18)& (2, 8, 10, 13)& (2, 8, 10, 19)\\ \hline (2, 8, 10, 22)& (2, 8, 11, 16)& (2, 8, 13, 16)& (2, 8, 14, 19)& (2, 8, 18, 19)\\ \hline (2, 9, 14, 16)& (2, 9, 15, 18)& (2, 9, 17, 22)& (2, 10, 13, 18)& (2, 10, 19, 22)\\ \hline (2, 11, 12, 15)& (2, 11, 15, 18)& (2, 11, 15, 20)& (2, 11, 16, 18)& (2, 12, 17, 22)\\ \hline (2, 13, 16, 17)& (2, 13, 18, 19)& (2, 14, 19, 20)& (2, 16, 17, 21)& (2, 17, 19, 22)\\ \hline (3, 4, 6, 7)& (3, 4, 6, 22)& (3, 4, 10, 11)& (3, 4, 14, 16)& (3, 4, 14, 21)\\ \hline (3, 5, 10, 18)& (3, 5, 10, 20)& (3, 5, 14, 17)& (3, 5, 18, 22)& (3, 6, 7, 17)\\ \hline (3, 6, 9, 20)& (3, 6, 14, 17)& (3, 6, 14, 20)& (3, 6, 15, 22)& (3, 6, 16, 17)\\ \hline (3, 6, 16, 20)& (3, 6, 18, 22)& (3, 7, 8, 12)& (3, 7, 9, 19)& (3, 7, 10, 11)\\ \hline (3, 7, 10, 12)& (3, 7, 17, 19)& (3, 7, 19, 22)& (3, 9, 10, 12)& (3, 9, 12, 19)\\ \hline (3, 9, 17, 20)& (3, 9, 19, 21)& (3, 10, 12, 20)& (3, 10, 13, 18)& (3, 10, 18, 20)\\ \hline (3, 11, 12, 15)& (3, 11, 13, 14)& (3, 11, 13, 16)& (3, 11, 14, 16)& (3, 11, 15, 21)\\ \hline (3, 12, 13, 16)& (3, 12, 15, 16)& (3, 12, 15, 20)& (3, 12, 16, 20)& (3, 14, 17, 20)\\ \hline (3, 15, 18, 21)& (3, 16, 17, 20)& (4, 5, 6, 7)& (4, 5, 8, 18)& (4, 5, 18, 19)\\ \hline (4, 6, 8, 11)& (4, 6, 8, 17)& (4, 6, 11, 22)& (4, 6, 12, 15)& (4, 6, 12, 17)\\ \hline (4, 6, 15, 17)& (4, 7, 8, 14)& (4, 7, 9, 19)& (4, 7, 19, 20)& (4, 8, 11, 17)\\ \hline (4, 8, 13, 14)& (4, 8, 13, 21)& (4, 8, 18, 21)& (4, 9, 10, 11)& (4, 9, 10, 19)\\ \hline (4, 9, 13, 19)& (4, 9, 13, 20)& (4, 9, 16, 20)& (4, 10, 14, 15)& (4, 10, 18, 21)\\ \hline (4, 11, 12, 19)& (4, 12, 13, 14)& (4, 12, 13, 20)& (4, 12, 15, 17)& (4, 13, 21, 22)\\ \hline (4, 14, 15, 16)& (4, 16, 17, 20)& (4, 16, 20, 22)& (4, 17, 20, 22)& (4, 17, 21, 22)\\ \hline (4, 20, 21, 22)& (5, 6, 7, 11)& (5, 6, 9, 20)& (5, 6, 10, 12)& (5, 6, 10, 18)\\ \hline (5, 6, 11, 13)& (5, 6, 11, 18)& (5, 6, 12, 18)& (5, 7, 11, 21)& (5, 7, 14, 22)\\ \hline (5, 7, 16, 22)& (5, 8, 9, 15)& (5, 8, 9, 18)& (5, 8, 9, 21)& (5, 8, 11, 21)\\ \hline (5, 8, 14, 15)& (5, 8, 15, 19)& (5, 8, 18, 19)& (5, 9, 11, 18)& (5, 9, 12, 16)\\ \hline (5, 9, 12, 18)& (5, 9, 15, 22)& (5, 9, 16, 21)& (5, 9, 16, 22)& (5, 9, 18, 22)\\ \hline (5, 9, 20, 21)& (5, 10, 17, 18)& (5, 10, 19, 21)& (5, 11, 17, 18)& (5, 12, 17, 18)\\ \hline  
        
        (5, 13, 16, 17)& (5, 14, 17, 22)& (5, 15, 19, 21)& (5, 15, 20, 22)& (6, 7, 11, 13)\\ \hline 
        (6, 7, 13, 21)& (6, 7, 17, 21)& (6, 8, 11, 13)& (6, 8, 12, 17)& (6, 8, 12, 19)\\ \hline (6, 8, 13, 14)& (6, 8, 15, 19)& (6, 9, 15, 22)& (6, 10, 12, 16)& (6, 10, 14, 17)\\ \hline (6, 10, 16, 18)& (6, 11, 15, 19)& (6, 12, 15, 17)& (6, 12, 16, 21)& (6, 13, 14, 17)\\ \hline (6, 14, 18, 22)& (6, 16, 17, 21)& (6, 16, 19, 20)& (6, 18, 19, 22)& (6, 19, 21, 22)\\ \hline (7, 8, 10, 12)& (7, 8, 13, 16)& (7, 8, 13, 21)& (7, 8, 15, 16)& (7, 9, 14, 19)\\ \hline (7, 9, 14, 21)& (7, 9, 15, 19)& (7, 10, 11, 15)& (7, 10, 11, 22)& (7, 11, 13, 20)\\ \hline (7, 11, 15, 20)& (7, 12, 13, 18)& (7, 12, 13, 20)& (7, 12, 13, 22)& (7, 12, 14, 20)\\ \hline (7, 12, 17, 18)& (7, 13, 15, 16)& (7, 13, 16, 22)& (7, 13, 18, 21)& (7, 14, 19, 20)\\ \hline (7, 16, 18, 22)& (7, 18, 19, 22)& (8, 9, 10, 17)& (8, 9, 10, 19)& (8, 9, 15, 19)\\ \hline (8, 9, 16, 19)& (8, 10, 12, 19)& (8, 10, 13, 22)& (8, 11, 13, 16)& (8, 11, 16, 20)\\ \hline (8, 12, 17, 19)& (8, 12, 18, 21)& (8, 12, 20, 21)& (8, 13, 14, 22)& (8, 14, 15, 19)\\ \hline (8, 14, 18, 22)& (8, 18, 20, 22)& (9, 10, 12, 19)& (9, 11, 12, 18)& (9, 11, 13, 20)\\ \hline (9, 12, 13, 14)& (9, 12, 16, 22)& (9, 13, 14, 17)& (9, 13, 14, 19)& (9, 13, 15, 19)\\ \hline (9, 14, 16, 19)& (9, 14, 16, 21)& (9, 14, 17, 22)& (10, 11, 16, 18)& (10, 11, 16, 20)\\ \hline (10, 11, 16, 22)& (10, 11, 17, 18)& (10, 12, 14, 20)& (10, 12, 16, 20)& (10, 13, 14, 17)\\ \hline (10, 13, 17, 18)& (10, 13, 18, 20)& (10, 14, 19, 20)& (10, 15, 16, 21)& (10, 16, 18, 21)\\ \hline (10, 16, 20, 22)& (10, 19, 20, 22)& (11, 12, 14, 18)& (11, 12, 15, 21)& (11, 13, 14, 19)\\ \hline (11, 14, 16, 18)& (11, 15, 17, 19)& (12, 13, 15, 16)& (12, 13, 15, 17)& (12, 13, 16, 22)\\ \hline (12, 13, 17, 18)& (12, 15, 16, 21)& (12, 17, 19, 22)& (13, 15, 19, 21)& (13, 15, 21, 22)\\ \hline (14, 15, 16, 21)& (14, 15, 18, 21)& (14, 16, 18, 21)& (14, 17, 21, 22)& (19, 20, 21, 22)\\ \hline
\end{longtable}

\newpage


\Addresses

\end{document}